\documentclass[10pt,psamsfonts]{amsart}
\usepackage[utf8]{inputenc}
\usepackage[dvips]{psfrag}
\usepackage{array}
\usepackage{color}
\usepackage{epsfig}
\usepackage{url}
\usepackage[hidelinks]{hyperref}
\usepackage{overpic}
\usepackage{epstopdf}
\usepackage{graphics}
\usepackage[all]{xy}
\usepackage{tikz, pgfplots}
\usepackage{subcaption}
\pgfplotsset{compat=1.12}
\usepgfplotslibrary{fillbetween}
\usepackage[shortlabels]{enumitem}
\usepackage{pythonhighlight}
\usepackage{comment}
\usepackage{amsmath}
\usepackage{amsthm}
\usepackage{amssymb}
\usepackage{amscd}
\usepackage{amsfonts}
\usepackage{amsbsy}
\usepackage{graphicx}
\usepackage[dvips]{psfrag}
\usepackage{array}
\usepackage{color}
\usepackage{tikz}
\usepackage{xcolor}
\usepackage{epsfig}
\usepackage{url}
\usepackage{hyperref}
\usetikzlibrary{arrows,intersections,calc}
\usepackage{float}
\usepackage{stmaryrd}	
\usepackage{overpic}
\usepackage{epstopdf}
\usepackage{epsfig,afterpage}
\usepackage[all]{xy}
\usepackage[makeroom]{cancel}
\usepackage{soul}
\usepackage{pgf,tikz}
\usepackage{mathrsfs}
\usetikzlibrary{arrows}
\usepackage{amsmath}
\usepackage{amsthm}
\usepackage{amssymb}
\usepackage{amscd}
\usepackage{amssymb}
\usepackage{amsfonts}
\usepackage{amsbsy}
\usepackage{amsaddr}
\usepackage{graphicx}
\usepackage{array}
\usepackage{color}
\usepackage{url}
\usepackage{overpic}
\makeatletter
\DeclareFontFamily{U}{tipa}{}
\DeclareFontShape{U}{tipa}{m}{n}{<->tipa10}{}
\newcommand{\ark@char}{{\usefont{U}{tipa}{m}{n}\symbol{62}}}%

\newcommand{\ark}[1]{\mathpalette\ark@arc{$#1$}}

\newcommand{\ark@arc}[2]{%
	\sbox0{$\m@th#1#2$}%
	\vbox{
		\hbox{\resizebox{\wd0}{\height}{\ark@char}}
		\nointerlineskip
		\box0
	}%
}
\makeatother
\newcommand{\R}{\ensuremath{\mathbb{R}}}

\newtheorem {theorem} {Theorem} %[section]

\newtheorem {definition} {Definition}
\newtheorem {remark} {Remark}
\newtheorem {example} {Example}

\begin{document}
	\renewcommand{\arraystretch}{1.5}
	
	\title[Shadowing Theorem for piecewise smooth vector fields]
	{About the Shadowing Theorem for piecewise smooth vector fields with sliding motion} 
	\author[T. Carvalho]
	{Tiago Carvalho$^1$}

	\address{$^1$ Departamento de  Matem\'{a}tica,  Instituto de Bioci\^{e}ncias, Letras e Ci\^{e}ncias Exatas,
		Universidade Estadual Paulista UNESP, Rua Crist\'{o}v\~{a}o Colombo, 2265, Zip Code 15054-000, S\~{a}o Jos\'{e} do Rio Preto, SP, Brazil.}\email{tiago.carvalho1@unesp.br}

	
	%34A36 Discontinuous ordinary differential equations
	
	%34A45 Theoretical approximation of solutions to ordinary differential equations
	
	%34A60 Ordinary differential inclusions
	
	%37C05 Dynamical systems involving smooth mappings and diffeomorphisms
	
	%37C15 Topological and differentiable equivalence, conjugacy, moduli, classification of dynamical systems
	
	%37C40 Smooth ergodic theory, invariant measures for smooth dynamical systems
	
	%34F05 Ordinary differential equations and systems with randomness
	
	%34C28 Complex behavior and chaotic systems of ordinary differential equations
	
	%37B65 Approximate trajectories, pseudotrajectories,shadowing and related notions for topological dynamical systems
	
	%37C50 Approximate trajectories (pseudotrajectories,shadowing, etc.) in smooth dynamics
	
	%37D45 Strange attractors, chaotic dynamics of systems with hyperbolic behavior

	
	\maketitle
	
	\begin{abstract}

%texto do abstract revisado IA

Since every modeling process in real-world situations is subject to errors, the study of the so-called shadowing property becomes highly relevant. This property allows for the identification of true orbits that closely follow a chain of trajectories in which some degree of approximation is introduced at each step. The objective of this paper is to establish a version of the Shadowing Theorem for vector fields governed by ordinary differential equations in the piecewise-smooth setting, where the dynamics alternate between distinct regimes, or “on–off” stages. First, we present an example showing that such a result cannot be obtained under the same hypotheses commonly assumed in the C$^2$ scenario. Consequently, under appropriate assumptions, we prove a Shadowing-like Theorem specifically tailored to this framework. Furthermore, we propose several extensions of the main result aimed at capturing characteristic phenomena intrinsic to piecewise-smooth vector fields - features that do not arise in the  C$^2$ setting.

	\end{abstract}

	%\linenumbers
	
	\section{Introduction}\label{intro}

Since mathematicians and scientists are continuously seeking mathematical models to describe real-world phenomena, it is crucial to understand under which conditions the dynamics generated by such models $-$ typically only approximations of the underlying processes $-$ are robust under small perturbations of initial conditions and reproducible across different computational implementations.

In numerical investigations, the evolution of an initial condition is represented by a sequence of iterates produced by successive evaluations of the model. Each iterate is inherently affected by rounding and truncation errors, and subsequent evaluations may be carried out from points that differ slightly from those previously computed, for instance when irrational values must be approximated by rational numbers within finite-precision arithmetic.

In the practice, what happens is the following (see Figure \ref{Fig cadeia de trajetorias}):  given a point $y_0$ as the initial condition, let us consider the trajectory  of a vector field $X$ passing through $y_0$ at time $h_{-1}=0$. In this case,  $y_0 = \varphi(0,y_0)$ where $\varphi(.,.)$ denotes the flow of $X$. 
The trajectory is calculated using a numerical method and, after a time  $h_0$, it is   stated that $y_1 = \varphi(h_0,y_0)$. However, errors were made during the process, and point $y_1$ is in fact an approximation of the exact value of $\varphi(h_0,y_0)$. However, the numerical method takes point $y_1$ as a new initial condition and, after a time  $h_1$, (inaccurately) states that $y_2 = \varphi(h_1,y_1)=\varphi(h_0 + h_1,y_0)$. The procedure is iterated until the point $y_N$ is reached. % (in a finite procedure $-$ see Remark \ref{remark tempo infinito}). 

	\begin{figure}[h]
	\begin{center}
		\begin{overpic}[width=4in]{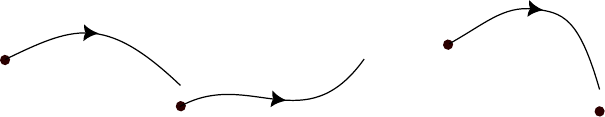}
		%	\begin{overpic}[grid,tics=10,width=4in]{cadeiaTraj.pdf}
			\put(-1,6){$y_0$}
				\put(14,16){$\Gamma_1$}
			\put(28,-2){$y_1$}
				\put(45,6){$\Gamma_2$}
			\put(63,10){$\hdots$}
			\put(72,8){$y_{N-1}$}
				\put(85,12){$\Gamma_N$}
			\put(98,-3){$y_N$}
		\end{overpic}
	\end{center}
	\caption{A (finite) chain of trajectories.}\label{Fig cadeia de trajetorias}
\end{figure}
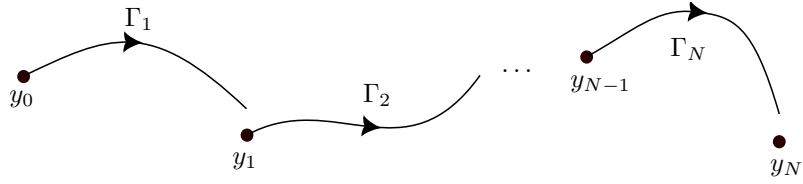

Given the notation presented above, let $\Gamma_i = \{ \varphi(h,y_{i-1}) \, | \, h \in [h_{i-2},h_{i-1}]\}$, with $i = 1,2, \hdots, N$, be an arc of trajectory of $X$ $-$ an exact one, not a numerically obtained. 
	A \textbf{chain of trajectories}  determined by the non-equilibria points $\{ y_k \}_{k=0}^{N}$ of a vector field $X$ is the set $ \bigcup_{i=1}^N \Gamma_i$. The numbers $h_i$, with $i=0,1, \hdots, N-1$ are the \textbf{timesteps}.

%
%\begin{remark}\label{remark tempo infinito}
%Previously we obtained a finite sequence $\{ y_k \}_{k=0}^{N}$. Using an analogous approach, a bi-infinite sequence $\{ y_k \}_{- \infty}^{+ \infty}$ is also obtained. As we will show below, the previous literature presents  slightly distinct results for finite and infinite sequences of points $y_k$. 
%\end{remark}
%
%
%In the sequel we will denote $\Z$ the set of integer numbers.

\begin{definition}
A chain of trajectories  determined by the non-equilibria points $\{ y_k \}_{k=0}^{N}$  %(respectively, $\{ y_k \}_{- \infty}^{+ \infty}$) 
of a vector field $X$ is a \textbf{$\delta$-pseudo orbit} when for a  real number $\delta>0$, it holds that 
\[ \| y_{k+1} - \varphi(h_k,y_k) \| \leq \delta, \mbox{ for all } k \in \{ 0 , 1 , \hdots, N-1 \} %(\mbox{respectively, } k \in \Z),
 \]where $\{ h_k \}_{k=0}^{N-1}$ %(respectively, $\{ h_k \}_{ - \infty}^{+ \infty}$)
  is a sequence of positive times and $\varphi(.,y_k)$ is a trajectory of $X$ that passes through $y_k \in \R^n$ at time $h_{k-1}$, with $h_{-1}=0$. 
\end{definition}

In this paper, we investigate sufficient conditions under which a PSVF satisfies the shadowing property. As stated in the abstract of \cite{SurveyShadowing}: \textit{a shadow is an exact solution to a set of equations that remains close to a numerical solution for a long time. Shadowing can thus be used as a form of backward error analysis for numerical solutions to ODEs.} Formally, we consider the following definition:

\begin{definition}	A $\delta$-pseudo orbit determined by $\{ y_i \}_{i=0}^{N}$ %(respectively, $\{ y_k \}_{- \infty}^{+ \infty}$) 
	 with timesteps  $\{ h_i \}_{i=0}^{N-1}$ %(respectively, $\{ h_k \}_{- \infty}^{+ \infty}$)
	  is \textbf{$\varepsilon$-shadowed} by a real trajectory  passing through $\{ x_i \}_{i=0}^{N}$ %(respectively, $\{ x_i \}_{- \infty}^{+ \infty}$)
	   with timesteps  $\{ \tau_i \}_{i=0}^{N-1}$ %(respectively, $\{ \tau_i \}_{- \infty}^{+ \infty}$)
	    when $x_{i+1} = \varphi(\tau_i,x_i)$ for some trajectory of $X$, where $\| y_i - x_i \| \leq \varepsilon$ and $|h_i - \tau_i| \leq \varepsilon$. 
\end{definition}

Several results concerning the Shadowing Theorem can be obtained in the context of discrete dynamical systems (see \cite{SurveyShadowing} and references therein). However, recently, some versions of the Shadowing Theorem were obtained in distinct contexts. For example, \cite{Backes-2022} obtained a version for nonautonomous and nonlinear models with impulses, \cite{Zhan-2020} obtained a version for stochastic differential equations, \cite{Gao-2021}
obtained a version for random dynamical systems. In the present paper we search for a version of the previous Shadowing Theorem for the context of \textbf{piecewise smooth vector fields} (PSVFs for short) which, for the best of our knowledge, is a missing gap in the literature. The development of a general  theory concerning PSVFs has been established in recent years not only because of the beauty of the theoretical results but also due to the value of it in applied sciences. Nowadays it is well known that a big variety of real phenomena are modeled using ODEs  that are subjected to a sudden and intermittent change in their description. It occurs in models describing the cancer dynamics where the model has a description for the period when the treatment is implemented (chemotherapy, for example) and a distinct description for the period when the treatment is interrupted for the body to recover from side effects (see \cite{Carvalho-leukemia,Carvalho-typicalSingCancer}). %The same abrupt and intermittent change on the model is observed in recent treatment models of HIV (see \cite{Ananworanich:2006,CarCrisGonTon-HIV-NoDy-2020,Tang:2015,Tang:2012,Carvalho-HiV-Daniel}), in power electronic (see \cite{CarCrisPagTon-PhysicaD-2017, Rony-IJBC-TS,  Kousaka}), predator-prey models (see \cite{CarNovGonShilnikov, Krivan1, Piltz}), COVID-19 (see \cite{Covid-Science,Covid-NoDy}) among other where a kind of \textit{"on-off"} state is considered on the model (see also  \cite{EGA2013, Hek2010, KruSzm2001, LlibrePonce2012, Rossa, IEEE-market}). 

In general, for the piecewise smooth scenario,  there exists a \textbf{control function} $f$ determining when the  \textit{"on-off"} state is considered. This means that for $f(z) \geq 0$ the model is governed by a vector field $X^{+}(z)$ and for $f(z) \leq 0$ the model is governed by another vector field $X^{-}(z)$. Then, the \textbf{switching manifold} $\Sigma = f^{-1}(0)$ plays a very important role in the analysis.

Now, the same behavior observed for the smooth scenario can occur here, i.e., there exists a piecewise smooth model that, due to approximations, presents a $\delta$-pseudo orbit. The question is:

\vspace{.3cm}
\begin{center}
	\textit{Under what hypotheses can a $\delta$-pseudo trajectory of a PSVF be $\varepsilon$-shadowed by a true orbit of the system ? }
\end{center}
\vspace{.3cm}

In Appendix we present a classical version of the Shadowing Theorem (Theorem \ref{teo shadowing para suaves}), stated in \cite{coomes-1994-FiniteChain}.  What we will require first is that, both vector fields $X^{\pm}$,  satisfy the hypotheses of Theorem \ref{teo shadowing para suaves}, because we hope to obtain the shadowing for each sub-chain of trajectories of the PSVFs. But, is it enough ? As we shall see in Example \ref{exemplo nao da certo} the answer is NO and some extra hypotheses must to be imposed for the PSVF. This is done in our main result,  Theorem \ref{teoShadowing}. 

After proving the main result in Theorem \ref{teoShadowing}, that extend to PSVFs the classical Shadowing Theorem, the dynamics around a PSVF is explored in-depth in order to reveal  some typical situations and we prove Theorems \ref{teoShadowingAnexo1}, \ref{teoShadowingAnexo2} and \ref{teoShadowingAnexo3} that will break, little by little, the similarities with the classic case.

The paper is organized as follows: In Section \ref{secao teoria basica CVSPs} we present the basic theory about PSVFs. We stress that here we define for the first time in the literature the original notion of \textbf{saturation of a recursive point} for PSVFs. Moreover, we give examples of PSVFs satisfying such a definition.  In Section \ref{secao main result} we announce and prove the main result of the paper, which is a \textbf{Shadowing-Like Theorem  for PSVFs defined in $\R^n$}. Section \ref{secao estendidos} is devoted to prove extensions  of the main result. In Section \ref{secao conclusao} we provide  some conclusions about the results proved in the paper are exposed. In the Appendix we announce the Shadowing Theorem given in \cite{coomes-1994-FiniteChain}.

	\section{Preliminaries}\label{secao teoria basica CVSPs}
	
Here we will state the basic definitions, conventions and notations concerning the theory about PSVFs. The first important task is the domain of the vector field. Let $V$ be an open set of $\mathbb{R}^n$ and consider  $\Sigma \subset V$ a  codimension-$1$ manifold in $\mathbb{R}^n$ given by $\Sigma = f^{-1} (0)= \left\lbrace q \in V : f(q) = 0 \right\rbrace $, where $f : V \rightarrow \mathbb{R} $ is a smooth function having $0 \in \mathbb{R}$ as a regular value (that is, $\nabla f(p) \neq 0 $, for $p \in f^{-1} (0)$). We call $\Sigma$ the \textbf{switching manifold} whose boundary separates the regions $\Sigma^{+} = \left\lbrace q \in V : f(q) \geq 0 \right\rbrace $ and $\Sigma^{-} = \left\lbrace q \in V : f(q) \leq 0 \right\rbrace $. 
	Let $\mathfrak{X}^r$ be the space of  $C^r$-vector fields in $V \subset \mathbb{R}^n$ endowed with the
	$C^r$-topology, with $r \geq 1$ large enough for our purposes. Call $\mathcal{Z}^r$ the space of PSVFs $Z : V \rightarrow \mathbb{R}^n $ such that
		\begin{equation}
		\label{sis}
		Z (q) = \left\{ \begin{array}{c}
			X^{+} (q) \, \, if \,\, q \, \in \Sigma^{+} \\
			X^{-}(q) \, \, if \,\, q \, \in \Sigma^{-}
		\end{array} \right. ,
	\end{equation}
	where $X^{+} = \left( X_1^{+}, X_2^{+}, \cdots ,  X_n^{+}\right) $, $X^{-} = \left(  X_1^{-}, X_2^{-}, \cdots ,  X_n^{-} \right) \in \mathfrak{X}^ r .$ We denote \eqref{sis} simply by $Z = (X^{+}, X^{-})$ when there is no confusion about the switching manifold and the function $f$ defining it.

Since we are considering the $C^r$-topology in  $\mathfrak{X}^r$, we get  the classical $C^r$-norm $\mid \cdot \mid_{C^r}$ of the  vector fields  $X^{+}$ and $X^{-}$ restricted to $\Sigma^{+}$ and $\Sigma^{-}$, respectively. In fact,  
	\[  \mid X^{\pm} \mid_{C^r}  = max_{|j| \leq r} sup_{x \in \Sigma^{\pm}} | D^j X^{\pm}(x) |, \]where $j=(j_1, \hdots,j_n)$ is a vector representing the order in the  mixed partial derivative $D^j X^{\pm}$  of $X^{\pm}$; moreover, $|j| = j_1 + \hdots + j_n$.  Note that $Z$ is multi-valued in $\Sigma$ (see the seminal book \cite{Fi}). Then, for points in $\Sigma$,
	\[ \parallel Z \parallel_{C^r} = max \left\lbrace \mid X^{+} \mid_{C^r}, \mid X^{-} \mid_{C^r} \right\rbrace.\] % where $\mid \cdot \mid_{C^r}$ denotes the classical $C^r$-norm of the  vector fields $X^{+}$ and $X^{-}$ restricted to $\Sigma^{+}$ and $\Sigma^{-}$, respectively, \textcolor{blue}{given by}
	%\[ \mid ...  \mid_{C^r}   \]

	In order to define rigorously the flow of  $Z$ passing through a point $p \in V$, we 
	must understand  the contact between the vector fields $X^{+},$ $X^{-}$ and $\Sigma$. This contact is characterized using the Lie derivative $X^{+} f(q) = \left\langle \nabla f(q), X^{+}(q) \right\rangle $, where $\left\langle \cdot, \cdot \right\rangle $ is the usual inner product. %We also use higher order derivatives given by $(X^{+})^k f = {X^{+}} ((X^{+})^{k-1} f) = \left\langle \nabla (X^{+})^{k-1} f, X^+ \right\rangle$ with $k$ being a positive integer.  
	The following generic regions appear on $\Sigma$:
	\begin{enumerate}
		\item[$\bullet$] \textbf{Crossing Region:}  $\Sigma^{c} = \left\lbrace p \in \Sigma \mid X^{+}f(q)  X^{-}f(q) > 0 \right\rbrace $. In addition we denote  $$\Sigma^{c^+} = \left\lbrace p \in \Sigma \mid X^{+}f(q) > 0,  X^{-}f(q) > 0 \right\rbrace $$ and $$\Sigma^{c^-} = \left\lbrace p \in \Sigma \mid X^{+}f(q) < 0,  X^{-}f(q) < 0 \right\rbrace ;$$
		\item[$\bullet$] \textbf{Sliding Region:} $\Sigma^{s} = \left\lbrace p \in \Sigma \mid X^{+}f(q) < 0, X^{-}f(q) > 0 \right\rbrace $;
		\item[$\bullet$] \textbf{Escaping Region:} $\Sigma^{e} = \left\lbrace p \in \Sigma \mid X^{+}f(q) > 0, X ^{-}f(q) < 0 \right\rbrace $.
	\end{enumerate}
	
	Any $q \in \Sigma$ such that $X^{+}f(q).X^{-}f(q) = 0$ is called a boundary singularity. The boundary singularities can be of two types: (i) an equilibrium of $X^{+}$ or $X^{-}$ over $\Sigma$ or (ii) a point where a trajectory of $X^{+}$ or $X^{-}$ is tangent to $\Sigma$ (and it is not an equilibrium of  $X^{+}$ or $X^{-}$). In the second case, $X^{+}f(p) = 0$ means that the trajectory of $X^{+}$ passing through $p$ is tangent to $\Sigma$  and we say that $p$ is a \textbf{tangential singularity} (or tangency
	point) of $X^{+}$. The same for $X^{-}$. The set composed by the tangential singularities is denoted by $\Sigma^t$.

%	A tangential singularity $p \in \Sigma$ is a \textbf{fold point} of $X^{+}$ if $X^{+}f(p) = 0$, but $(X^{+})^2f(p) \neq 0 .$ Moreover, $p \in \Sigma$ is a visible (respectively, invisible) fold point of $X^{+}$ if $X^{+}f(p) = 0$ and $(X^{+})^2f(p) > 0$ (respectively, $(X^{+})^2f(p) < 0$). Analogously for $X^{-}$ reversing the last two inequalities.  When $p$ is a fold point for both
%	$X^{+}$ and $X^{-}$, we say that $p$ is a \textbf{fold–fold singularity} or \textbf{two-fold singularity}. A two-fold is called
%	\begin{enumerate}
%		\item[$1.$] visible-visible, if it is a visible tangency for both $X^{+}$ and $X^{-}$;
%		\item[$2.$] invisible-invisible, if it is an invisible tangency for $X^{+}$ and $X^{-}$;
%		\item[$3.$] visible-invisible, whether it is a visible tangency for $X^{+}$ and an invisible tangency for $X^{-}$ or vice versa.
%	\end{enumerate} 
	
	The trajectories of a PSVF passing through a crossing point are defined as the concatenation of the trajectories of $X^{+}$ and $X^{-}$ by that point since the vector fields $X^{+}$ and $X^{-}$ point to the same direction. However, in the sliding and escaping regions, we need to define an auxiliary vector field called sliding vector field, which is a convex linear combination of $X^{+}(p)$ and $X^{-}(p)$ in such a way that $Z^s$ is tangent to $\Sigma$ in the cone generated by $X^{+}(p)$ and $X^{-}(p).$ %	\begin{definition}
%		\label{deslizante}
%		Given a point $p \in \Sigma^{s} \cup \Sigma^{e} \subset \Sigma $, we define the \textbf{sliding vector field} at $p$ as the vector field $Z^s(p) = m - p$ with $m$ being the point of the segment joining $p + X^{+}(p)$ and $y + X^{-}(p)$ such that $m - p$ is tangent to $\Sigma$. See Figure \ref{campodeslizante}.
%	\end{definition}
	Explicitly, 
	
	\begin{definition}
	The \textbf{sliding vector field} is given by: 
		\begin{equation*}
		Z^s(q) = \frac{X^{-}f(q)X^{+}(q) - X^{+}f(q)X^{-}(q)}{X^{-}f(q) - X^{+}f(q)} \quad \textrm{for all} \; q \in \Sigma^{s} \cup \Sigma^{e} .
	\end{equation*}
	\end{definition}
	
	Moreover, when it is well defined (i.e., $\Sigma^{s} \cup \Sigma^{e} \neq \emptyset$), the sliding vector field can be extend to $\overline{\Sigma^{s}} \cup \overline{\Sigma^{e}}.$ A point $p \in \Sigma^{s} \cup \Sigma^{e}$ such that $Z^s(p) = 0$ is called a \textbf{pseudo equilibrium } of $Z$. 

\vspace{.5cm}

	\begin{definition}
		\label{definicao trajetorias locais}
The local trajectory (orbit) $\phi_Z (t, p)$ of a PSVF $Z = (X^{+}, X^{-})$ through a small neighborhood of $p \in U$ is defined as follows: \begin{enumerate}
\item[$(i)$] For $p \in \Sigma^{+} \setminus \Sigma$ and $p \in \Sigma^{-}\setminus\Sigma$ the trajectory is given by $\phi_Z (t, p) = \phi_{X^{+}} ( t,p)$ and $\phi_Z (t,p) = \phi_{X^{-}} (t,p)$ respectively. 
\item[$(ii)$] For $p \in \Sigma^{c^{+}}$ and taking the origin of time at $p$ the trajectory is defined as $\phi_Z (t, p) = \phi_{X^{-}} (t,p)$ for $t \leq 0$ and $\phi_Z (t,p) = \phi_{X^{+}} (t,p)$ for $t \geq 0.$ If $p \in \Sigma^{c^{-}}$ the definition is the same reversing the time;
\item[$(iii)$] For $p \in \Sigma^{e}$ and taking the origin of time at $p$, the trajectory is defined as $\phi_Z (t, p) = \phi_{Z^ s} (t,p)$ for $t \leq 0$ and $\phi_Z (t, p) $ is either $\phi_{X^{+}} (t, p)$ or $\phi_{X^{-}} (t, p)$ or $\phi_{Z^s} (t, p)$ for $t\geq 0.$ For $p \in \Sigma^s$ the definition is the same  reversing the time;
\item[$(iv)$] For $p$ being a  tangential singularity and taking the origin of time at $p$ the trajectory is defined as $\phi_Z (t, p) = \phi_1 (t,p)$ for $t \leq 0$ and $\phi_Z (t,p) = \phi_2 (t,p)$ for $t \geq 0,$ where each $\phi_1, \phi_2$ is either $\phi_{X^{+}}$ or $\phi_{X^{-}}$ or $\phi_{Z^s}$.% or $\phi_{Z^T}$ ;
%\item[$(v)$] For $p \in V \subset \mathbb{R}^2$ a singular tangency point, $\phi_Z (t, p) = p$ for all $t \in \mathbb {R}.$ In other cases, taking the origin of time at $p$, the trajectory is defined as $\phi_{z}(t,p) = \phi_{Z^T}(t,p)$.% for $t \in I .$
		\end{enumerate}
	\end{definition}

	\begin{definition}
		\label{def1}
A global trajectory $\Gamma_Z (t, p_0)$ of a PSVF is a concatenation of local trajectories. Moreover, a maximal trajectory  is a global trajectory that can not be extended to any other global trajectories by joining local ones, that is, if $\widetilde{\Gamma_Z} $ is a global trajectory containing $\Gamma_Z$ then $\Gamma_Z = \widetilde{\Gamma_Z}$. In this case, we call $I^{\Gamma_{Z}}_{max} = (\tau^{-}(\Gamma_{Z},p_0), \tau^{+} (\Gamma_{Z},p_0))$ the maximal interval of existence of the solution $\Gamma_Z$.% A global trajectory is a \textbf{positive} (respectively, \textbf{negative}) global trajectory if $t > 0$ (respectively, $t < 0)$ and $t_0 = 0$.
	\end{definition}

\begin{remark}
 The maximal interval of existence of a solution of $Z$ may not cover the interval $(- \infty , \infty)$, that is, $\tau^{\pm} (\Gamma_{Z},p_0)$ could assume finite values. When there is no danger of confusion, we will prefer use the notation $I_{max} = (\tau^{-}(p_0), \tau^{+} (p_0))$ instead of $I^{\Gamma_{Z}}_{max} = (\tau^{-}(\Gamma_{Z},p_0), \tau^{+} (\Gamma_{Z},p_0))$.
\end{remark}

The next example proves that, even if $X^{\pm}$ satisfy the hypotheses of Theorem \ref{teo shadowing para suaves}, the shadowing property can not be satisfied by the PSVF $Z=(X^+,X^-)$. 

%\begin{definition}\label{def star-shaped} Let $L$ be a set of indexes and $T_l > 0$, with $l \in L$.
%	We say that a set $\mathcal{A}$ is \textbf{star-shaped}  with center in $p \in \mathcal{A}$ when  every point $q \in \mathcal{A}$ is such that $q = \varphi_{Z}(t^{1}_q,p)=\varphi_{Z}(-t^{2}_q,p)$ with $0< t^{1}_q < T_{l_{1}}$ and  $- T_{l_{2}}< - t^{2}_q < 0$  for some $Z$-trajectory $\varphi_Z$, some $l_{1}, l_{2} \in L$ and $p = \varphi_{Z}(T_{l_1},p)=\varphi_{Z}(-T_{l_2},p)$.
%\end{definition}

\begin{example}\label{exemplo nao da certo}
	Consider the PSVF with phase portrait given in Figure \ref{Fig nao sombreamento}. Observe that, for all $\delta>0$, there exists a $\delta$-$Z$-pseudo orbit $\{ y_1,y_2, y_3 \}$. However, no true $Z$-orbit passes through $\varepsilon$-neighborhoods of the points $y_1,y_2,y_3$ simultaneously. In fact, given $N_{1}^{\varepsilon}$ a  $\varepsilon$-neighborhood of $y_1$ we observe that the local $Z$-trajectory $\gamma_1$ passing through $y_1$ divides it into two components $N_{1,+}^{\varepsilon}$ and  $N_{1,-}^{\varepsilon}$ in such a way that  $N_{1}^{\varepsilon} = N_{1,+}^{\varepsilon} \cup \gamma_1 \cup  N_{1,-}^{\varepsilon}$ is a disjoint union.  Just points of $\gamma_1 \cup  N_{1,-}^{\varepsilon}$ reaches $N_{2}^{\varepsilon}$, a $\varepsilon$-neighborhood of $y_2$. 
	
	Now, repeating the partition done before, $N_{2}^{\varepsilon} = N_{2,+}^{\varepsilon} \cup \gamma_2 \cup  N_{2,-}^{\varepsilon}$ is a disjoint union of $N_{2}^{\varepsilon}$, where $\gamma_2$ is a local $Z$-trajectory passing through $y_2$. Observe that, just points of $N_{2,-}^{\varepsilon}$ reaches $N_{3}^{\varepsilon}$, a $\varepsilon$-neighborhood of $y_3$. 
	
	Since $\gamma_1 \cup  N_{1,-}^{\varepsilon}$ reaches $N_{2}^{\varepsilon}$ in $N_{2,+}^{\varepsilon}$, there is not a true $Z$-orbit passing through $N_{1}^{\varepsilon}$, $N_{2}^{\varepsilon}$ and $N_{3}^{\varepsilon}$. This proves that the existence of $\delta$-$Z$-pseudo orbits is not enough to ensure the existence of $\varepsilon$-shadowing.

	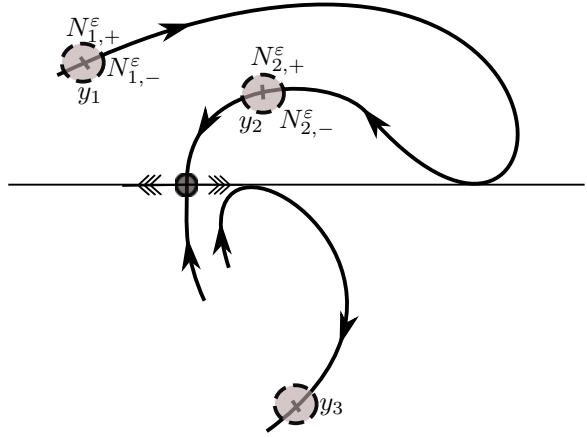
\begin{figure}[h]
		\begin{center}
			\begin{overpic}[width=3in]{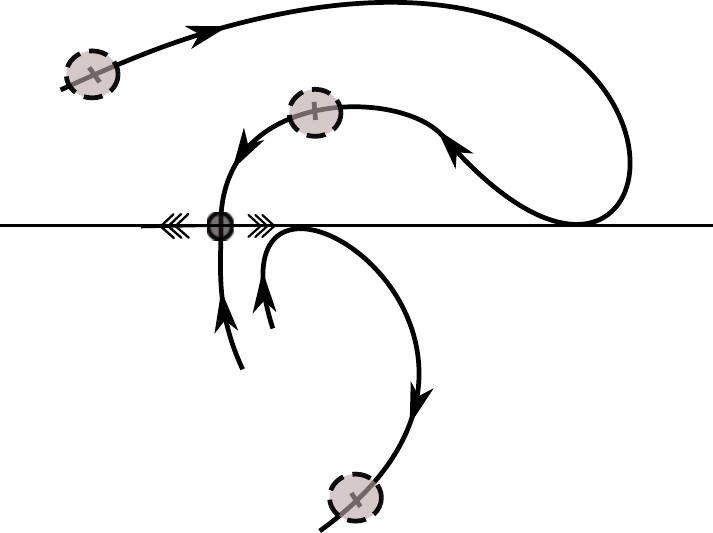}
				%\begin{overpic}[grid,tics=10,width=3in]{Fig-Nao-Sombreamento.pdf}
				\put(12,58){$y_1$}
				\put(40,53){$y_2$}
				\put(54,4){$y_3$}
				\put(10,69){$N^{\varepsilon}_{1,+}$}
				\put(17,62){$N^{\varepsilon}_{1,-}$}
				\put(42,64){$N^{\varepsilon}_{2,+}$}
				\put(47,53){$N^{\varepsilon}_{2,-}$}
			\end{overpic}
		\end{center}
		\caption{Figure of Example \ref{exemplo nao da certo}. Observe that there is not a real trajectory of the PSVF $Z$ connecting a small neighborhood of $y_1$ to a small neighborhood of $y_3$.}\label{Fig nao sombreamento}
	\end{figure}

\end{example}

To overcome this problem and obtain a $\varepsilon$-shadowing for PSVFs, we need to impose additional assumptions on the PSVFs. In this sense, considering the non-uniqueness of deterministic trajectories passing though a point $p \in \Sigma$,  we present the following set  that will play a central role in the paper:

\begin{definition}\label{def recursive point} A point  $p$ of the domain of the PSVF $Z = (X^+,X^-)$ is a \textbf{recursive point} when
 every $Z$-trajectory passing through $p$ return to $p$, both in forward and backward time.

Let $\mathcal{A}$ be the subset of the domain of $Z$ composed by all points that belong to, at least, one $Z$-trajectory passing through the recursive point $p$. We call  $\mathcal{A}$ the \textbf{saturation of  $p$}.

\end{definition}

%In other words, a star-shaped set $\mathcal{A}$ is composed by all points belonging to closed trajectories passing through a center point $p \in \mathcal{A}$.

\begin{remark}
	Observe that, if $p$ is a recursive point and $\mathcal{A}$  is the saturation of  $p$, there exists a  set  of indexes $L$ and $T_l > 0$, with $l \in L$ such that   every point $q \in \mathcal{A}$ is such that $q = \varphi_{Z}(t^{1}_q,p)=\varphi_{Z}(-t^{2}_q,p)$ with $0< t^{1}_q < T_{l_{1}}$ and  $- T_{l_{2}}< - t^{2}_q < 0$  for some $Z$-trajectory $\varphi_Z$, some $l_{1}, l_{2} \in L$ and $p = \varphi_{Z}(T_{l_1},p)=\varphi_{Z}(-T_{l_2},p)$. This ensures that there is a closed trajectory passing through $q$, with return time equals to $t^{1}_q + t^{2}_q$.
\end{remark}

Before moving on, let us show examples concerning the existence of recursive points and its saturation.

\begin{example}\label{exemplo feijao}
	In  \cite{BCEchaotic} it was firstly  considered   the following planar PSVF:
	\begin{equation}\label{eq feijao}
		Z(x,y)= \left\lbrace \begin{array}{ll}
			X(x,y) = (	1, -2x)
			, & y \geq 0 \\
			Y(x,y) = \left(  -2,- 4 x^3 + 2 x  \right) , & y \leq 0.
		\end{array} \right. \end{equation}
	This PSVF presents a non trivial minimal set with positive measure (see the main result of \cite{BCEchaotic}).  Moreover, in \cite{JDE-Entropia-2023} it was proved that the entropy in this non trivial minimal set is infinite. Also, in 	 \cite{Tiago-ClosingLema} and \cite{CarvalhoEuzebio-EJQTDE} the same set is studied in distinct contexts. 
	
	Using Definition \ref{def recursive point}, the set in Figure \ref{Fig feijao} is the saturation of the recursive point placed at the origin. 
	
\begin{figure}[h]
	\begin{center}
		\begin{overpic}[width=2.5in]{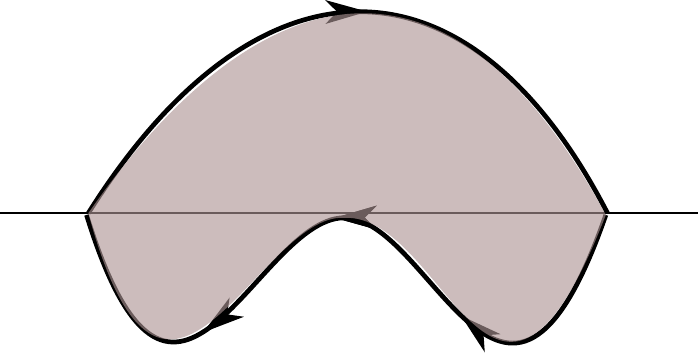}
			%\begin{overpic}[grid,tics=10,width=4.5in]{fig-dinamicaglobal.pdf}
			\put(49,22){$0$}
			%\put(101,37){$T^*$}\put(20,11){$V_I$}\put(50,10){$\Sigma$} \put(61,58){$r^{V}_{X}$}\put(30,33){$r^{V}_{Y}$}\put(52,51){$C_X$}\put(32,44){$C_Y$}\put(45,45){$TF$}\put(60,29){$S(r^{V}_{Y})$}\put(35,63){$S(r^{V}_{X})$}%\put(,){$$}\put(,){$$}\put(,){$$}
		\end{overpic}
	\end{center}
	\caption{Saturation of a recursive point placed at the origin,  given by System \eqref{eq feijao} (see also  \cite{BCEchaotic}). Note that the saturation set has positive Lebesgue measure in $\R^n$.}\label{Fig feijao}
\end{figure}

\end{example}

\begin{example}\label{exemplo petalas e figura oito}
	Consider either the planar PSVF considered in Section 5 of \cite{JDE-Entropia-2023} (see Figure \ref{Fig petalas}) or the PSVF  $Z(x,y) = (X(x,y),Y(x,y))$ presented in the end of \cite{CarvalhoEuzebio-EJQTDE}, where  $X(x, y) = (1, 4x(1 -x^2))$, $Y(x, y) =
(-1, 4x(1 - x^2))$ and the switching manifold being the $x$-axis (see Figure \ref{Fig caos trivial}). The set of points belonging to the trajectories passing through  the origin is a saturation of the recursive point placed at the origin. 
	
	\begin{figure}[h]
		\begin{center}
			\begin{overpic}[width=2.5in]{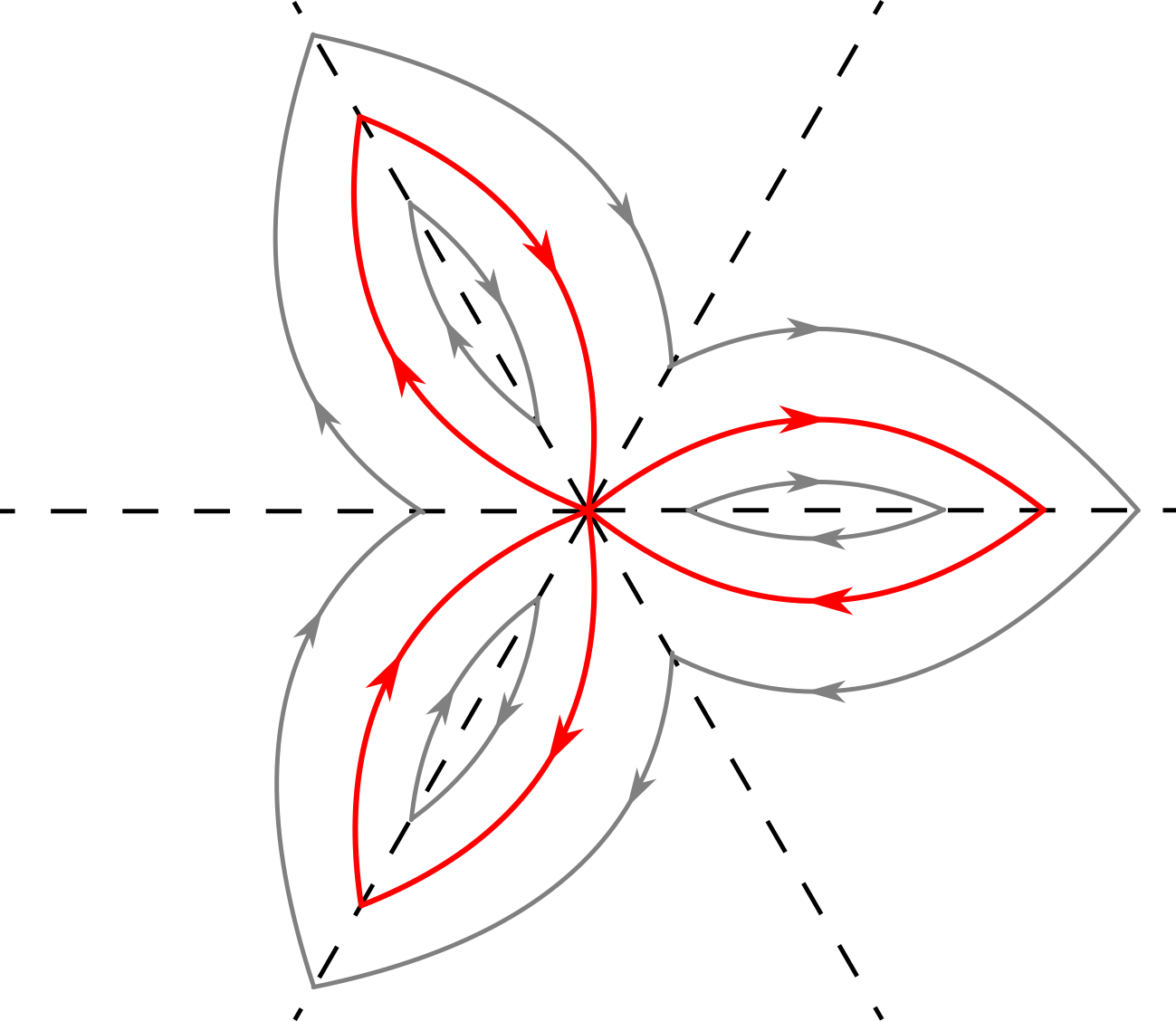}
				%\begin{overpic}[grid,tics=10,width=4.5in]{fig-dinamicaglobal.pdf}
				%\put(50,92){$T$}\put(101,37){$T^*$}\put(20,11){$V_I$}\put(50,10){$\Sigma$} \put(61,58){$r^{V}_{X}$}\put(30,33){$r^{V}_{Y}$}\put(52,51){$C_X$}\put(32,44){$C_Y$}\put(45,45){$TF$}\put(60,29){$S(r^{V}_{Y})$}\put(35,63){$S(r^{V}_{X})$}%\put(,){$$}\put(,){$$}\put(,){$$}
			\end{overpic}
		\end{center}
		\caption{Saturation of the recursive point placed at the origin,  given in Section 5 of \cite{JDE-Entropia-2023}. Note that the saturation set has  Lebesgue measure equals to zero in $\R^n$.}\label{Fig petalas}
	\end{figure}
	
		\begin{figure}[h]
		\begin{center}
			\begin{overpic}[width=2in]{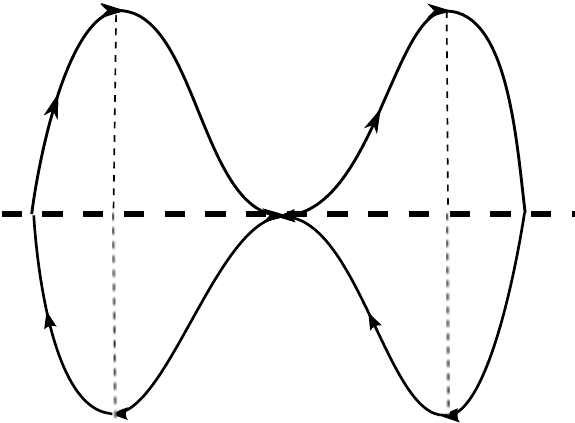}
				%\begin{overpic}[grid,tics=10,width=4.5in]{fig-dinamicaglobal.pdf}
				%\put(50,92){$T$}\put(101,37){$T^*$}\put(20,11){$V_I$}\put(50,10){$\Sigma$} \put(61,58){$r^{V}_{X}$}\put(30,33){$r^{V}_{Y}$}\put(52,51){$C_X$}\put(32,44){$C_Y$}\put(45,45){$TF$}\put(60,29){$S(r^{V}_{Y})$}\put(35,63){$S(r^{V}_{X})$}%\put(,){$$}\put(,){$$}\put(,){$$}
			\end{overpic}
		\end{center}
		\caption{Saturation of the recursive point placed at the origin, given in \cite{CarvalhoEuzebio-EJQTDE}. Note that the saturation set has  Lebesgue measure equals to zero in $\R^n$.}\label{Fig caos trivial}
	\end{figure}

\end{example}

%
%\begin{definition}
%	A pseudo trajectory $\{ y_i \}_{i=0}^{N}$ with timesteps  $\{ h_i \}_{i=0}^{N-1}$ is $\varepsilon$-shadowed by an real trajectory $\{ x_i \}_{i=0}^{N}$ with timesteps  $\{ \tau_i \}_{i=0}^{N-1}$ when $x_{i+1} = \varphi_{Z}(\tau_i,x_i)$ for some $Z$-trajectory, where $\| y_i - x_i \| \leq \varepsilon$ and $|h_i - \tau_i| \leq \varepsilon$. 
%\end{definition}

\section{Main result}\label{secao main result}

Now we state the main result of the paper:

\begin{theorem}[Shadowing-Like Theorem  for PSVFs defined in $\R^n$]
\label{teoShadowing} 
Consider an $n$-dimensional  PSVF $Z=(X^{+},X^{-}) \in \mathcal{Z}^r$ defined in the saturation $\mathcal{A}$ of the recursive point $p^*$ of Z. Suppose that both $X^{+}$ and $X^{-}$ satisfy the hypotheses of   Theorem \ref{teo shadowing para suaves}. 

Then, given  a 
sufficiently small $\varepsilon > 0$, there is a $\delta > 0$  such that
 any $\delta$-$Z$-pseudo-orbit
$\{ y_k \}_{k=0}^{N} \subset \mathcal{A}$, with $y_k \notin \Sigma$ for all $k \in \{ 0, 1, \hdots, N \}$, is $\varepsilon$-shadowed by a  true $Z$-orbit containing
points $\{ x_k \}_{k=0}^{N}$. %Moreover, there is only one such sequence satisfying the
%orthogonality condition $\left\langle x_k - y_k, Z(y_k) \right\rangle = 0 $. 
\end{theorem}
\begin{proof}
Take a  sufficiently small $\varepsilon > 0$ and consider $\{ y_k \}_{k=0}^{N} \subset \mathcal{A}$ a $\delta$-$Z$-pseudo-orbit, with $y_k \notin \Sigma$ for all $k \in \{ 0, 1, \hdots, N \}$. 

If the trajectories connecting the successive $y_i$ do not intercept $\Sigma$, we conclude that either $\{ y_k \}_{k=0}^{N} \subset \Sigma^+$ or $\{ y_k \}_{k=0}^{N} \subset \Sigma^-$. In both cases, since both $X^{+}$ and $X^{-}$ satisfy  Theorem \ref{teo shadowing para suaves}, we obtain that $\{ y_k \}_{k=0}^{N}$ is $\varepsilon$-shadowed by a true orbit $\{ x_k \}_{k=0}^{N}$. % and the orthogonality condition also is verified. 

 If the trajectories connecting the successive $y_i$  intercept $\Sigma$, the situation is a bit complex. In fact,  $\{ y_k \}_{k=0}^{N}$ can oscillate between $\Sigma^+$ and  $\Sigma^-$ or even $\{ y_k \}_{k=0}^{N}$ is composed by subsets $\{ y_k \}_{k=0}^{N} = \{ y_{k} \}_{k=0}^{N_{k_1}} \cup \{ y_{k} \}_{k=N_{k_1}}^{N_{k_2}} \cup \hdots \cup \{ y_{k} \}_{k=N_{k_{n}}}^{N}$ in such a way that the trajectory connecting $y_{k_{m}}$ and the $\delta$-neighborhood of $y_{k_{m+1}}$ intersects $\Sigma$, for $m=1,2, \hdots, n$. In the sequel, we will detail the proof for two representative cases: $(i)$:   $\{ y_k \}_{k=0}^{N} = \{ y_{k} \}_{k=0}^{N_{k_1}} \cup \{ y_{k} \}_{k=N_{k_1}}^{N} $, where $\{ y_{k} \}_{k=0}^{N_{k_1}} \subset  \Sigma^+$ and  $\{ y_{k} \}_{k=N_{k_1}}^{N} \subset \Sigma^-$ and $(ii)$ $\{ y_k \}_{k=0}^{N} = \{ y_{k} \}_{k=0}^{N_{k_1}} \cup \{ y_{k} \}_{k=N_{k_1}}^{N} $, where both  $\{ y_{k} \}_{k=0}^{N_{k_1}}$ and  $\{ y_{k} \}_{k=N_{k_1}}^{N}$ are contained in $ \Sigma^+$ but the trajectory connecting $y_{N_{k_{1}}}$ and the $\delta$-neighborhood of $y_{N_{k_{1}}+1}$ intersects $\Sigma$. The other cases are very similar and the proofs will be omitted.

$\bullet$ \textbf{Connection between a compound in $\Sigma^+$ and a compound in $\Sigma^-$:}  Call $\{ y_{k} \}_{k=0}^{N_{k_1}} =  \{ y^{+}_k \}_{k=0}^{N^+}$ and $\{ y_{k} \}_{k=N_{k_1}}^{N} =  \{ y^{-}_k \}_{k=0}^{N^-}$.

Since $X^+$ satisfies Theorem \ref{teo shadowing para suaves}, there exists $\delta^+ > 0$ such that  any $\delta^+$-$X^+$-pseudo-orbit
$\{ y^{+}_k \}_{k=0}^{N^+} \subset \mathcal{A}$, with $y^{+}_k \notin \Sigma$ for all $k \in \{ 0, 1, \hdots, N^+ \}$, is $\varepsilon$-shadowed by a  true orbit of $X^+$ containing
points $\{ x^{+}_k \}_{k=0}^{N^+}$. %Moreover, there is only one such sequence satisfying the
%orthogonality condition $\left\langle x^{+}_k - y^{+}_k, X^{+}(y^{+}_k) \right\rangle = 0 $. 
The same holds for $X^-$, changing the symbol "$+$" by "$-$".

Take $\delta  = max \{ \delta^-, \delta^+ \}$, $N=N^{-} + N^{+} + 1$ and consider the $\delta$-$Z$-pseudo orbit $\{ y_k \}_{k=0}^{N} = \{ y^{+}_0, y^{+}_1, \hdots, y^{+}_{N^{+}}, y^{-}_0, y^{-}_1, \hdots, y^{-}_{N^{-}} \}$.

Since $\mathcal{A}$ is the saturation of $p^*$,  there exist $t^{+}_{1} > 0$ and $t^{-}_{1} > 0$ such that $\varphi_{Z}(t^+,x^{+}_{N^{+}}) = \varphi_{Z}(-t^-,x^{-}_{0}) = p^* \in \mathcal{A}$ for some concatenation $\gamma_{1}$ of local $Z$-trajectories.

Then, using $\gamma_{1}$ to connect $x^{+}_{N^{+}}$ and $x^{-}_{0}$, we conclude the existence of a true orbit of $Z$ passing through   $\{ x_k \}_{k=0}^{N} = \{ x^{+}_0, x^{+}_1, \hdots, x^{+}_{N^{+}}, x^{-}_0, x^{-}_1, \hdots, x^{-}_{N^{-}} \}$.  

%The uniqueness of, for example $\{ x^{+}_k \}_{k=0}^{N^{+}}$, ensures the uniqueness of $\{ x_k \}_{k=0}^{N}$ such that $\left\langle x_k - y_k, Z(y_k) \right\rangle = 0 $. 

%Analogously, there exist $t^{+}_2 > 0$ and $t^{-}_2 > 0$ such that $\varphi_{Z}(t^{-}_2,y^{-}_{N^{-}}) = \varphi_{Z}(-t^{+}_2,y^{+}_{0}) = p_{2}^* \in \Sigma$ for some concatenation of local $Z$-trajectories.

$\bullet$ \textbf{Connection between a compound in $\Sigma^+$ and another compound in $\Sigma^+$:}  Again, call $\{ y_{k} \}_{k=0}^{N_{k_1}} =  \{ y^{+}_k \}_{k=0}^{N^+}$ and $\{ y_{k} \}_{k=N_{k_1}}^{N} =  \{ y^{-}_k \}_{k=0}^{N^-}$.

Since $X^+$ satisfies Theorem \ref{teo shadowing para suaves}, there exists $\delta^{+} > 0$ such that  any $\delta^+$-$X^+$-pseudo-orbit
$\{ y^{+}_k \}_{k=0}^{N^+} \subset \mathcal{A}$ (respectively,  $\{ y^{-}_k \}_{k=0}^{N^-}$), with $y^{+}_k \notin \Sigma$ for all  $k \in \{ 0, 1, \hdots, N^+   \}$ (respectively, $y^{-}_k$ and   $N^-$), is $\varepsilon$-shadowed by a  true orbit of $X^+$ containing
points $\{ x^{+}_k \}_{k=0}^{N^+}$ (respectively, $\{ x^{-}_k \}_{k=0}^{N^-}$). %Moreover, there is only one such sequence satisfying the
%orthogonality condition $\left\langle x^{\pm}_k - y^{\pm}_k, X^{+}(y^{\pm}_k) \right\rangle = 0 $.

Take $\delta = \delta^+$, $N=N^{-} + N^{+} + 1$ and consider the $\delta$-$Z$-pseudo orbit \linebreak $\{ y_k \}_{k=0}^{N} = \{ y^{+}_0, y^{+}_1, \hdots, y^{+}_{N^{+}}, y^{-}_0, y^{-}_1, \hdots, y^{-}_{N^{-}} \}$.

Since $\mathcal{A}$ is the saturation of $p^*$,  there exist $t^{+}_{1} > 0$ and $t^{-}_{1} > 0$ such that $\varphi_{Z}(t^+,x^{+}_{N^{+}}) = \varphi_{Z}(-t^-,x^{-}_{0}) = p^* \in \mathcal{A}$ for some concatenation $\gamma_{1}$ of local $Z$-trajectories.

Then, using $\gamma_{1}$ to connect $x^{+}_{N^{+}}$ and $x^{-}_{0}$, we conclude the existence of a true orbit of $Z$ passing through   $\{ x_k \}_{k=0}^{N} = \{ x^{+}_0, x^{+}_1, \hdots, x^{+}_{N^{+}}, x^{-}_0, x^{-}_1, \hdots, x^{-}_{N^{-}} \}$.  

%The uniqueness of, for example $\{ x^{+}_k \}_{k=0}^{N^{+}}$, ensures the uniqueness of $\{ x_k \}_{k=0}^{N}$ such that $\left\langle x_k - y_k, Z(y_k) \right\rangle = 0 $. 

\end{proof}

\begin{remark}
	The PSVFs  of Examples \ref{exemplo feijao} and  \ref{exemplo petalas e figura oito} satisfy Theorem \ref{teoShadowing}.
\end{remark}

\section{Extended shadowing theorems}\label{secao estendidos}

While Theorem \ref{teoShadowing} extends to PSVFs the classical Shadowing Theorem, the dynamics around a PSVF is even complex and a more in-depth study deserves to be done to reveal some typical behavior.

In the sequel we will write a sequence of theorems that will break, little by little, the similarities with the classic case. We invite the reader to this journey, where prejudices need to be put aside.

\begin{theorem}
	\label{teoShadowingAnexo1} 
	Consider the same hypotheses of Theorem \ref{teoShadowing}. 
 	Then, given  a 
	sufficiently small $\varepsilon > 0$, there is a $\delta > 0$  such that given two distinct $\delta$-$Z$-pseudo-orbits
		$\{ y_k \}_{k=0}^{N} \subset \mathcal{A}$ and $\{ z_m \}_{m=0}^{M} \subset \mathcal{A}$, with $y_k,z_m \notin \Sigma$ for all $k,m$, the set  $\{ y_k \}_{k=0}^{N} \cup \{ z_m \}_{m=0}^{M}$  is $\varepsilon$-shadowed by a  true $Z$-orbit containing
		points $\{ x_k \}_{k=0}^{N} \cup \{ w_m \}_{m=0}^{M} $. %Moreover, there is only one such sequence satisfying the
	%	orthogonality condition $\left\langle x_k - y_k, Z(y_k) \right\rangle = 0 $ and  $\left\langle w_m - z_m, Z(z_m) \right\rangle = 0 $. 
	The same holds taking a finite  arbitrary  number of $\delta$-$Z$-pseudo-orbits.

\end{theorem}
\begin{remark}
Theorem \ref{teoShadowingAnexo1} says that even if $\{ y_k \}_{k=0}^{N} \subset \mathcal{A}$ and $\{ z_m \}_{m=0}^{M} \subset \mathcal{A}$ are, a priori, \textbf{disjoint} $\delta$-pseudo orbits, it is possible to obtain a true orbit $\varepsilon$-shadowing both of them. This fact does not happens for  C$^2$ vector fields.
\end{remark}
\begin{proof}

Take a  sufficiently small $\varepsilon > 0$ and consider two distinct $\delta$-$Z$pseudo-orbits
$\{ y_k \}_{k=0}^{N} \subset \mathcal{A}$ and $\{ z_m \}_{m=0}^{M} \subset \mathcal{A}$, with $y_k,z_m \notin \Sigma$ for all $k,m$.   

As we said at the proof of Theorem \ref{teoShadowing}, there are a lot of possibilities for the distribution of the points composing the pseudo-orbits in $\Sigma^+$ and $\Sigma^-$.  Without loss of generality,  we will suppose that 
 $\{ y_k \}_{k=0}^{N} = \{ y^{+}_k \}_{k=0}^{N^+} \cup \{ y^{-}_k \}_{k=0}^{N^-}$, where $\{ y^{+}_k \}_{k=0}^{N^+} \subset  \Sigma^+$ and  $\{ y^{-}_k \}_{k=0}^{N^-} \subset \Sigma^-$ and  $\{ z_m \}_{m=0}^{M}= \{ z^{1}_{k} \}_{k=0}^{N^{1}} \cup \{ z^{2}_{k} \}_{k=0}^{N^{2}} \subset  \Sigma^+$ in such a way that the trajectory connecting $z^{1}_{N^{1}}$ and the $\delta$-neighborhood of $z^{2}_{0}$ intersects $\Sigma$.

Consider exactly the same construction and the same true $Z$-orbits that are $\varepsilon$-shadowing $\{ y_k \}_{k=0}^{N}$ and $\{ z_m \}_{m=0}^{M}$ presented in the proof of Theorem \ref{teoShadowing}. See the \textit{Connection between a compound in $\Sigma^+$ and a compound in $\Sigma^-$} and the \textit{Connection between a compound in $\Sigma^+$ and another compound in $\Sigma^+$}.

So, there exist $\delta^{1} > 0$ and $\delta^{2} > 0$ such that $\{ w^{1}_m \}_{m=0}^{N}$ is $\varepsilon$-shadowing the  $\delta^{1}$-pseudo orbit $\{ y_{k} \}_{k=0}^{N}$ and  $\{ w^{2}_m \}_{m=0}^{M}$ is $\varepsilon$-shadowing the  $\delta^{2}$-pseudo orbit $\{ z_{k} \}_{k=0}^{M}$. % in such a way that there is only one such sequences satisfying the
%orthogonality condition $\left\langle w^{1}_k - y_k, Z(y_k) \right\rangle = 0 $ and  $\left\langle w^{2}_k - z_k, Z(z_k) \right\rangle = 0$. 

Now, let us connect these two true $Z$-orbits. Using the definition of the connecting domain $\mathcal{A}$, there exist $t^{+}_{2} > 0$ and $t^{-}_{2} > 0$ such that $\varphi_{Z}(t^{-}_{2},w^{1}_{N}) = \varphi_{Z}(-t^{+}_{2},w^{2}_{0}) = p_{2}^* \in \Sigma$ for some concatenation $\gamma_{2}^*$ of local $Z$-trajectories.

Take $\delta = max \{ \delta^1, \delta^2 \}$ and consider the $\delta$-$Z$-pseudo orbit \linebreak $\{ \widetilde{y}_k \}_{k=0}^{N+M+1} = (y_0, y_1, \hdots, y_{N}, z_0, z_1, \hdots, z_{M})$.

Take $\{ w_k \}_{k=0}^{N+M+1} = (w^{1}_0, w^{1}_1, \hdots, w^{1}_{N}, w^{2}_0, w^{2}_1, \hdots, w^{2}_{M})$. Then, using $\gamma_{2}^*$, we conclude that  there exists a true orbit of $Z$ passing through $\{ w_k \}_{k=0}^{N+M+1}$. %, that is the only one presenting the orthogonality condition.

%The uniqueness of $\{ x^{+}_k \}_{k=0}^{N^{+}}$ and $\{ w^{1}_k \}_{k=0}^{N^{1}}$ ensure the uniqueness of $\{ x_k \}_{k=0}^{N}$ such that $\left\langle x_k - y_k, Z(y_k) \right\rangle = 0 $ and $\{ w_k \}_{k=0}^{N}$ such that $\left\langle w_k - z_k, Z(z_k) \right\rangle = 0 $, respectively. 
\end{proof}

\begin{theorem}
	\label{teoShadowingAnexo2} 
	Consider the same hypotheses of Theorem \ref{teoShadowing}.
		Then, the true orbits obtained in Theorems \ref{teoShadowing} and \ref{teoShadowingAnexo1} can be taken closed/periodic.  %, if the number of $\delta$-pseudo orbits of Theorem \ref{teoShadowingAnexo1} is finite, can be taken closed/periodic.
		\end{theorem}
\begin{proof}
	In this case it is enough to observe that, from the definition of  saturation of a recursive point, the orbit departing from the last point of the true orbit obtained in the proofs of Theorems \ref{teoShadowing} and \ref{teoShadowingAnexo1} must reach the recursive point $p^*$.

	If we consider the true orbit departing from $p^*$ and performing the same path already described in the proofs of Theorems \ref{teoShadowing} and \ref{teoShadowingAnexo1}, we will have the required closed/periodic true orbit. 
\end{proof}

\begin{theorem}
	\label{teoShadowingAnexo3} 
	Consider the same hypotheses of Theorem \ref{teoShadowing}. 
	Then, the true orbits obtained in Theorems \ref{teoShadowing} and \ref{teoShadowingAnexo1} can be taken
dense in the closure of  $\mathcal{A}$.
	\end{theorem}
\begin{proof}
 Consider the true orbit $\gamma$ obtained in Theorems \ref{teoShadowing} and \ref{teoShadowingAnexo1}.

Let $p \in \overline{\mathcal{A}}$, the closure of $\mathcal{A}$ and consider a small open ball $B_{\mu}^{p}$ centered in $p$ with ray $\mu>0$.

The true orbit  $\gamma$ return to  the  recursive point $p^*$ of the  set $\mathcal{A}$. After it, from the definition of $\mathcal{A}$, there exists  an arc of trajectory connecting $p^*$ and $B_{\mu}^{p}$. 
\end{proof}

\section{Conclusion}\label{secao conclusao}

This paper establishes a version of the classical Shadowing Theorem within the framework of PSVFs for finite-time chain of trajectories. It demonstrates that, even when considering pseudo-trajectories, it is possible to obtain a true orbit arbitrarily close to the chain of trajectories composing the pseudo-orbit. This result is obtained under the assumption that the domain of the PSVF is the saturation of a recursive point, which is a notion introduced herein as an original contribution to the literature.

After the proof of the main result, three additional theorems are presented to further elucidate the distinctions between the smooth (or  C$^2$) and piecewise smooth scenarios. Specifically, these results establish: (i) that an arbitrary number of distinct pseudo-orbits can be simultaneously shadowed by a single true orbit (Theorem \ref{teoShadowingAnexo1}); (ii) that such shadowing can be realized in a manner where the true orbit is closed/periodic (Theorem \ref{teoShadowingAnexo2}); and (iii) that the shadowing can be performed so that the true orbit is dense in the domain of the PSVF (Theorem \ref{teoShadowingAnexo3}).

Particularly, the last 3 results  do not occur in the C$^2$ context. This reveals that the shadowing property produces surprising and amazing phenomena that are  typical situations obtained for PSVFs. % In the future, an even deeper study can be done using such shadowing property to obtain more typical behavior for PSVFs. 

\section*{Acknowledgements}

Tiago Carvalho is partially supported by S\~{a}o Paulo Research Foundation (FAPESP grants \# 2024/15612-6, \#2021/12395-6, \#2019/10269-3 and \#2022/02819-6) and by Conselho Nacional de Desenvolvimento Cient\'{i}fico e Tecnol\'{o}gico (CNPq Grants 309378/2023-0 and 401974/2025-1).

%\section*{Data Availability}
%
%Data sharing not applicable to this article as no datasets were generated or analysed during the current study.
%
%\section*{Statement}
%
%The  author states that there is no conflict of interest.

%%\bibliographystyle{unsrt}

% \bibliographystyle{elsarticle-num} 
%\bibliography{referencial-Tiago-out25.bib}

\begin{thebibliography}{10}
			\expandafter\ifx\csname url\endcsname\relax
			\def\url#1{\texttt{#1}}\fi
			\expandafter\ifx\csname urlprefix\endcsname\relax\def\urlprefix{URL }\fi
			\expandafter\ifx\csname href\endcsname\relax
			\def\href#1#2{#2} \def\path#1{#1}\fi
			
			\bibitem{SurveyShadowing}
			W.~Hayes, K.~R. Jackson, A survey of shadowing methods for numerical solutions
			of ordinary differential equations, Applied Numerical Mathematics 53 (2005)
			299--321.
			
			\bibitem{Backes-2022}
			L.~Backes, D.~Dragičević, L.~Singh, Shadowing for nonautonomous and nonlinear
			dynamics with impulses, Monatsh Math 198 (2022) 485--502.
			
			\bibitem{Zhan-2020}
			Q.~Zhan, Z.~Zhang, Y.~Li, Numerical implementation of finite-time shadowing of
			stochastic differential equations, Indian J Pure Appl Math 51 (2020) 1939 --
			1957.
			
			\bibitem{Gao-2021}
			S.~Gao, A shadowing lemma for random dynamical systems, Journal of Applied
			Analysis and Computation 11~(6) (2021) 3014 -- 3030.
			
			\bibitem{Carvalho-leukemia}
			L.~F. Gon\c{c}alves, D.~S. Rodrigues, P.~F.~A. Mancera, T.~Carvalho, A
			mathematical model for chemoimmunotherapy of chronic lymphocytic leukemia,
			Applied Mathematics and Computation 349 (2019) 118--133.
			
			\bibitem{Carvalho-typicalSingCancer}
			L.~F. Gon\c{c}alves, D.~S. Rodrigues, P.~F.~A. Mancera, T.~Carvalho, Sliding
			mode control in a mathematical model to chemoimmunotherapy: the occurrence of
			typical singularities, Applied Mathematics and Computation 387 (2020) 124782.
			
			\bibitem{coomes-1994-FiniteChain}
			B.~A. Coomes, H.~Ko\'{c}ak, K.~J. Palmer, Shadowing orbits of ordinary
			differential equations, Journal of Computational and Applied Mathematics 52
			(1995) 35--43.
			
			\bibitem{Fi}
			A.~F. Filippov, Differential Equations with Discontinuous Righthand Sides, 1st
			Edition, Vol.~18 of Mathematics and its Applications, Springer Netherlands,
			1988.
			
			\bibitem{BCEchaotic}
			C.~A. Buzzi, T.~Carvalho, R.~D. Euz\'{e}bio, Chaotic planar piecewise smooth
			vector fields with non-trivial minimal sets, Ergodic Theory and Dynamical
			Systems 36~(2) (2016) 458--469.
			
			\bibitem{JDE-Entropia-2023}
			T.~Carvalho, A.~Antunes, R.~Var\~{a}o, On topological entropy of piecewise
			smooth vector fields, Journal of Differential Equations 362 (2023) 52--73.
			
			\bibitem{Tiago-ClosingLema}
			T.~Carvalho, On the closing lemma for planar piecewise smooth vector fields,
			Journal de Math\`{e}matiques Pures et Appliqu\'{e}es 106~(6) (2016)
			1174--1185.
			
			\bibitem{CarvalhoEuzebio-EJQTDE}
			T.~Carvalho, R.~Euzebio, Minimal sets and chaos in planar piecewise smooth
			vector fields, Electronic Journal of Qualitative Theory of Differential
			Equations 33 (2020) 1--15.
			
		\end{thebibliography}

	\section*{Appendix}
	
	The purpose of this appendix is to announce the Shadowing Theorem stated in \cite{coomes-1994-FiniteChain}. To do this, we rewrite the introductory technical text given in \cite{coomes-1994-FiniteChain}.
	
Consider a C$^2$ vector field $\dot{x} = F(x)$.	Let $\{y_k\}_{k=0}^N$ be a $\delta$-pseudo-orbit of $(1)$ with associated times
	$\{h_k\}_{k=0}^{N-1}$. Also suppose that we have a sequence
	$\{Y_k\}_{k=0}^{N-1}$ of $n \times n$ matrices such that
	\begin{equation}
		\|Y_k - DF(y_k)\| < \delta, \quad \text{for } k = 0, \ldots, N-1.
		\tag{2}
	\end{equation}
	
	We will define a sequence $\{A_k\}_{k=0}^{N-1}$ of $(n-1)\times(n-1)$ matrices
	in the following way. For $k = 0, \ldots, N$, let $S_k$ be an $n \times (n-1)$
	matrix chosen so that its columns form an orthonormal basis for the subspace
	orthogonal to $F(y_k)$. Now, we let
	\begin{equation}
		A_k = S_{k+1}^* Y_k S_k, \quad \text{for } k = 0, \ldots, N-1,
		\tag{3}
	\end{equation}
	where $^*$ denotes transpose. Geometrically, $A_k$ is $Y_k$ restricted to the
	subspace orthogonal to $F(y_k)$ and then projected onto the subspace orthogonal
	to $F(y_{k+1})$.
	
	Next, we define a linear operator
	\[
	L : (\mathbb{R}^{n-1})^{N+1} \to (\mathbb{R}^{n-1})^{N}
	\]
	in the following way. If
	$\xi = \{\xi_k\}_{k=0}^N \in (\mathbb{R}^{n-1})^{N+1}$, then we take
	$L\xi = \{(L\xi)_k\}_{k=0}^{N-1}$ to be
	\begin{equation}
		(L\xi)_k = \xi_{k+1} - A_k \xi_k, \quad \text{for } k = 0, \ldots, N-1.
		\tag{4}
	\end{equation}
	
	The operator $L$ has right inverses, and we choose one such right inverse $L^{-1}$.
	
	As the last piece of our notational collection, we define several constants.
	Let $\varepsilon_0$ be a positive number and let $U$ be a convex open set
	containing $\{y_k\}_{k=0}^N$ such that if $x$ is in the ball about $y_k$
	with radius $\varepsilon_0$, then the solution $\phi^t(x)$ is defined for
	$0 < t < 2h_k$ and is in $U$. For such a $U$, we define
	\[
	M_0 = \sup_{x \in U} \| F(x)\| \,\, , \,\,	M_1 = \sup_{x \in U} \|D F(x)\| \,\, , \,\,	M_2 = \sup_{x \in U} \|D^2 F(x)\|.
	\]
	Finally, we define
	\[
\Delta = \inf_{0 \leq k \leq N} \| F(y_k)\| \,\, , \,\, \Theta = \sup_{0 \leq k \leq N - 1} \| Y_k\| \,\, , \,\, 
	h = \sup_{0 \le k \le N-1} h_k.
	\]
	
	Now, we can state our main theorem.

	\begin{theorem}\label{teo shadowing para suaves}
	[Finite-time Shadowing Theorem \cite{coomes-1994-FiniteChain}]
Let $\{y_k\}_{k=0}^N$ be a $\delta$-pseudo-orbit of the autonomous system
$\dot{x} = F(x)$, and let
\[
C = \max \left\{ \Delta^{-1}\bigl( \Theta \|L^{-1}\| + 1\bigr), \, \|L^{-1}\| \right\}.
\]
If $\delta$ satisfies the inequalities
\begin{enumerate}
\item[(i)]	$C(M_1 + 1)\delta \le \frac{1}{2},$

	\item[(ii)]
	$
	4C\delta < \min_{0 \le k \le N-1} h_k,
	\qquad
	4C\delta < \varepsilon_0,
	$
	
	\item[(iii)]
	$
	8\bigl(M_0 M_1 + 2 M_1 e^{2 M_1 h}
	+ 2 M_2 h e^{4 M_1 h}\bigr) C^2 \delta \leq 1,
	$
\end{enumerate}
then the pseudo-orbit $\{y_k\}_{k=0}^N$ is $\varepsilon$-shadowed by a true orbit
$\{x_k\}_{k=0}^N$ with $\varepsilon < 4C\delta$.
	\end{theorem}

			\end{document}